\documentclass[11pt]{amsart}
\usepackage{amssymb, latexsym, amsmath, amsfonts, tikz, tikz-cd}
\usepackage{hyperref, enumitem, fancyhdr}

\newtheorem{thm}{Theorem}[section]

\newtheorem{lem}[thm]{Lemma}

\theoremstyle{definition}

\theoremstyle{remark}

\numberwithin{equation}{section}
\theoremstyle{remark}

\newcommand{\mbb}{\mathbb}
\newcommand{\ra}{\rightarrow}

\newcommand{\pa}{\partial}

\newcommand{\sm}{\setminus}

\newcommand{\no}{\noindent}

\newcommand{\cal}{\mathcal}

\begin{document}
\title{Non-negative divisors and the Grauert metric}
\keywords{Grauert metric, holomorphic sectional curvature, non-negative divisors}
\thanks{SG is supported by CSIR-SPM Ph.D. fellowship}
\subjclass{Primary: 32T05, 32Q15}
\author{Sahil Gehlawat and Kaushal Verma}

\address{SG: Department of Mathematics, Indian Institute of Science, Bangalore 560 012, India}
\email{sahilg@iisc.ac.in}

\address{KV: Department of Mathematics, Indian Institute of Science, Bangalore 560 012, India}
\email{kverma@iisc.ac.in}

\begin{abstract}
Grauert showed that it is possible to construct complete K\"{a}hler metrics on the complement of complex analytic sets in a domain of holomorphy. In this note, we study the holomorphic sectional curvatures of such metrics on the complement of a principal divisor in $\mbb{C}^n$, $n \ge 1$. In addition, we also study how this metric and its holomorphic sectional curvature behaves when the corresponding principal divisors vary continuously.
\end{abstract}

\maketitle 

%%%%%%%%%%%%%%%%%%%%%
\section{Introduction}

\noindent According to Grauert \cite{Gr}, it is possible to construct a complete K\"{a}hler metric on the complement of a complex analytic set in a domain of holomorphy. The purpose of this note is twofold. First, we study the holomorphic sectional curvatures of the metric that is obtained by this process on the complement of a principal divisor in $\mbb C^n$, $n \ge 2$. The case $n =1$ have been studied in detail in the last section. Secondly, we show that if $U \subset \mbb C^n$, $n \ge 1$ is a domain of holomorphy (or more generally, admits a complete K\"{a}hler metric) with $H^1(U, \mbb Z) = 0$, then Grauert's construction possesses an intrinsic continuity property when the principal divisors vary continuously in an appropriate sense. 

\medskip

To set things in context, recall that the set of divisors on a complex manifold $M$ is 
\[
\mathcal D(M) = \left\{ \sum_{i \in I} n_{i} V_{i} : n_i \in \mbb Z  \right\}
\]
where $\{V_i\}_{i \in I}$ is a locally finite system of irreducible analytic hypersurfaces in $M$ and the support of a divisor $D \in \mathcal D(M)$ is the union of the hypersurfaces $V_i$ -- this will be denoted by $\vert D \vert$.  The set of non-negative divisors $\mathcal D^+(M) \subset \mathcal D(M)$ consists of those divisors wherein $n_i \ge 0$ for all $i \in I$. The natural map 
\[
{\sf Div} : \mathcal O(M)  \ra \mathcal D^+(M)
\]
defined by 
\[
{\sf Div}(f) =\sum_i \left( {\rm ord}_{V_i}f \right) V_i
\]
assigns a holomorphic function to the divisor defined by its zero set; here, $\{V_i : i \in I\}$ are the irreducible components of $\mathcal Z(f)$, the zero set of a holomorphic function $f \in \mathcal O(M)$, and ${\rm ord}_{V_i}f$ is the order of $V_i \subset \mathcal Z(f)$. The set of {\it principal divisors} arises as the image ${\sf Div}(\mathcal O(M)) \subset \mathcal D^+(M)$ and this will be denoted by $\mathcal D^+_P(M)$.

\medskip

\no Let $D \in \mathcal{D}_{P}^{+}(\mbb{C}^n)$ be a principal divisor. Choose $f \in \cal{O}(\mbb{C}^n)$ such that ${\sf Div}(f) = D$. Recall that the Grauert metric on $\mbb{C}^{*}$ (see \cite{Gr}) is a conformal metric defined by
\begin{equation}
g = \Big( 1 + \left\vert w \right\vert^2 u^{2}\big(\vert w\vert^2\big)\Big) \left\vert dw\right\vert^2
\end{equation}
where $u : (0, \infty) \rightarrow \mbb{R}$ is the function $u(t) = (t-1)/(t \log{t})$. It was shown in \cite{GV} that the curvature of $g$ approaches $-4$ or $0$ according as $\vert z \vert \ra 0$ or $+\infty$ respectively, and is non-positive everywhere on $\mbb C^{\ast}$. Consider the complete K\"ahler metric on $\mbb{C}^n \setminus \vert D \vert$ defined by
\begin{equation}
\phi(z,V) = f^{*}(g)(z,V) + \vert V \vert^2 = \Big( 1 + \left\vert f(z)\right\vert^2 u^{2}\big(\vert f(z)\vert^2\big)\Big) \left\vert df(z)(V)\right\vert^2 + \vert V\vert^2
\end{equation}
where $z \in \mbb{C}^n \sm \vert D \vert$ and $V \in \mbb{C}^n$.

\medskip

\no For $p \in \mbb{C}^n \sm \vert D \vert$ and $V \in \mbb{C}^n$, denote the holomorphic sectional curvature of $\phi$ at $(p, V)$ by $K(p, V)$. Set $K^+(p)$ to be the supremum of  $K(p, V)$ as $V$ varies in $\mbb C^n$. For a non-singular vector field $X$ near
$p \in \mbb{C}^n \sm \vert D \vert$, consider the restriction of $\phi$ to the leaf of the foliation defined by $X$ passing through $p$. Thus, we get a conformal metric on the germ of a Riemann surface containing $p$. Let $\cal{K}^{X}(p)$ denote the Gaussian curvature of this metric at $p$.

\begin{thm}
Let $(D,f,\phi)$ be as above on $\mbb{C}^n$, $n \ge 2$.
\begin{itemize}
\item[(i)] If $p \in \mbb{C}^n \setminus \vert D \vert$ and $V \in {\rm Ker}(df(p))$, then $K(p,V) \le 0$. In particular, if ${\rm Rank}(df(p)) = 0$, then $K^{+}(p) \le 0$.

\item[(ii)] If $p \in \mbb{C}^n \sm \vert D \vert$ and $f_{z_i}(p)  = \frac{\pa f}{\pa z_i}(p) \neq 0$ for some $1 \le i \le n$, then there exists a non-singular holomorphic vector field $X$ near $p$ such that $X(p) \notin {\rm Ker}(df(p))$ and $\cal{K}^{X}(p) \le 0$.

\item[(iii)] If $p \in \vert D \vert$ and $f_{z_i}(p) = \frac{\pa f}{\pa z_i}(p) \neq 0$, then there exists a non-singular holomorphic vector field $X$ near $p$ such that
\[
\lim_{z \to p} \cal{K}^{X}(z) = -4 < 0
\]
\end{itemize}
\end{thm}

\medskip

\no To summarize, for points away from the support of $D$, (i) shows that the curvature $K(p, X)$ is non-positive along certain directions. On the other hand, though the conclusions of (ii) and (iii) are less precise, they show that the curvature of the restriction of $\phi$ to the leaves of a certain foliation is again non-positive.

\medskip

\no It is natural to identify a non-negative divisor with the current of integration over it, namely:
\[
D(\alpha) = \langle \alpha, D\rangle = \Sigma_{j} m_{j} \int_{V_{j}} \alpha
\]
for each smooth compactly supported $(n-1, n-1)-$form $\alpha$ on $X$ where $D = \Sigma_{j} m_j V_j$. Thus $\cal{D}^{+}(M)$ can be considered as a subset of the dual space of $D^{(n-1,n-1)}(M)$, the space of compactly supported $(n-1,n-1)-$smooth forms on $M$. Therefore $\cal{D}^{+}(M)$ inherits two natural topologies namely, the relative weak* topology and the relative strong topology.

\medskip
\no \textbf{Weak* Topology}: A net $\{D_{\alpha}\}_{\alpha \in A}$ in $\cal{D}^{+}(M)$ converges to $D_{0} \in \cal{D}^{+}(M)$ iff for every compactly supported smooth $(n-1,n-1)-$form $\xi$ on $M$,
$$\lim_{\alpha \in A} \int_{D_{\alpha}} \xi = \int_{D_{0}} \xi$$ 

\medskip
\no \textbf{Strong Topology}: A net $\{D_{\alpha}\}_{\alpha \in A}$ in $\cal{D}^{+}(M)$ converges to $D_{0}\in \cal{D}^{+}(M)$ iff it converges in the weak* sense and if, moreover, the convergence is uniform on bounded sets in the space $D^{(n-1,n-1)}(M)$.

\medskip
\no There is another topology on $\cal{D}^{+}(M)$ introduced by Stoll \cite{St} which is defined in function-theoretic terms:

\medskip

\no \textbf{Stoll's Topology:} A net $\{D_{\alpha}\}_{\alpha \in A}$ in $\cal{D}^{+}(M)$ converges to $D_{0} \in \cal{D}^{+}(M)$ iff there is an open cover $\cal{V} = \{V_{j}\}_{j \in \mbb{N}}$ of $M$ such that for each $\alpha \in A$, there is an $f_{j \alpha} \in \cal{O}(V_{j})$ such that ${\sf Div}(f_{j \alpha}) = D_{\alpha}\vert_{V_{j}}$ and $f_{j \alpha}$ converges uniformly on compacts in $V_{j}$ to $f_{j0} \in \cal{O}(V_{j})$, where ${\sf Div}(f_{j 0}) = D_{0} \vert_{V_j}$.

\medskip
\no Lupacciolu-Stout \cite{LS} have shown that on $\cal{D}^{+}(M)$, the above three topologies are equivalent. Using this equivalence they showed: \textit{For a complex-analytic manifold $M$ of dimension $n \ge 1$ satisfying $H^{1}(M, \mbb{Z}) = 0$ and $H^{1}(M, \cal{O}) = 0$, there exist a continuous map $\psi : \cal{D}^{+}_{P}(M) \rightarrow \cal{O}(M)$ such that ${\sf Div}(\psi(D)) = D$ for all $D \in \cal{D}^{+}_{P}(M)$}. In addition, if $M$ is a domain in a Stein manifold $N$ satisfying $H^{1}(M, \mbb{Z}) = 0$, then there exist a continuous map $\psi: \cal{D}^{+}_{P}(M) \rightarrow \cal{O}(M)$ such that ${\sf Div}(\psi(D)) = D$ for all $D \in \cal{D}^{+}_{P}(M)$.

\medskip
\no Let $U \subset \mbb{C}^n$ be a domain of holomorphy or more generally a domain that has a complete K\"ahler metric $\Phi_{U}$. Let $D \in \cal{D}^{+}_{P}(U)$ be a principal divisor and $f \in \cal{O}(U)$ be a holomorphic function such that $D = {\sf Div}(f)$. Define a pseudometric on $U \sm \vert D \vert$ by $\tilde{\phi}_{D} := f^{*}(g)$, where $g$ is the Grauert metric on $\mbb{C}^{*}$. Therefore, $\phi_{D} = \tilde{\phi}_{D} + \Phi_{U}$ is a complete K\"ahler metric on $U \sm \vert D \vert$.

\medskip

The following statements clarify the dependence of the Grauert metric and its curvature as a function of the divisor $D$.

\begin{thm}
Let $U \subset \mbb{C}^n$, $n \ge 1$ be a domain as above, which also satisfies $H^{1}(U, \mbb{Z}) = 0$. If $\{D_{j}\}_{j \in \mbb{N}} \subset \cal{D}^{+}_{P}(U)$ is a sequence of non-negative principal divisors such that $D_{j}$ converges to $D_0 \in \cal{D}^{+}_{P}(U)$ with respect to any of the equivalent topologies above, then
\begin{itemize}
\item[(i)] there exists a sequence of complete K\"ahler metrics $\{\phi_{j}\}$ and $\phi_{0}$ on $U \sm \vert D_{j} \vert$ and $U \sm \vert D_{0} \vert$ repectively such that $\phi_{j}$ converges to $\phi_{0}$ uniformly on compacts of $U \sm \vert D_{0} \vert$.
\item[(ii)] for a fixed complete K\"ahler metric on $U \sm \vert D_{0} \vert$ of the form $\xi = f^{*}(g) + \Phi_{U}$ where $D_{0} = {\sf Div}(f)$, we can choose a sequence of complete K\"ahler metrics $\{\xi_{j}\}$ on $U \sm \vert D_{j} \vert$ such that $\xi_{j}$ converges to $\xi$ uniformly on compacts of $U \sm \vert D_{0} \vert$.
\end{itemize} 
\end{thm}

\medskip
\no Fix $p \in U \sm \vert D_{j} \vert$ and a non-singular holomorphic vector field $X$ near $p$. For a complete K\"ahler metric $\phi_{j}$ on $U \sm \vert D_{j} \vert$, denote the Gaussian curvature of the metric $\phi_{j}$ restricted to the leaves of foliation induced by $X$ by $\cal{K}^{X}_{\phi_{j}}(q)$ for points $q \in U \sm \vert D_{j} \vert$ near $p$. For $V \in \mbb{C}^n$, let $K_{\phi_{j}}(p, V)$ denote the holomorphic sectional curvature of the metric $\phi_{j}$ at $(p,V)$.

\begin{thm}
Let $(U, \phi_{j}, \phi_{0})$ be as in the above theorem and fix $p \in U \sm \vert D_{0} \vert$. Then
\begin{itemize}
\item[(i)] for a non-singular holomorphic vector field $X$ near $p$, there exists a neighourhood $U_p$ of $p$ in $U \sm \vert D_{0} \vert$ such that $\cal{K}^{X}_{\phi_{j}}(q) \rightarrow  \cal{K}^{X}_{\phi_0}(q)$ for all $q \in U_p$
\item[(ii)] $K_{\phi_0}(p, V) \le \liminf_{j \to \infty}{K_{\phi_{j}}(p, V)}$ for all $V \in \mbb{C}^n$.
\end{itemize}
\end{thm}

%%%%%%%%%%%%%%%%%%%%%%%%%%%%%%

\section{Proof of Theorem 1.1}
Fix a non-negative principal divisor $D \in \cal{D}^{+}_{P}(\mbb{C}^n)$ and $f \in \cal{O}(\mbb{C}^n)$ such that ${\sf Div}(f) = D$. The corresponding complete K\"ahler metric on $\mbb{C}^n \sm \vert D \vert$ is given by:
$$\phi(z,V) = \Big(1 + \left\vert f(z)\right\vert^2 u^{2}\big(\vert f(z)\vert^2\big)\Big) \left\vert df(z)(V)\right\vert^2 + \vert V \vert^2$$
where $z \in \mbb{C}^n \sm \vert D \vert$ and $V \in \mbb{C}^n$.

\medskip

\no For $(i)$, let $p \in \mbb{C}^n \sm \vert D \vert$ and let $X$ be a non-singular holomorphic vector field near $p$ such that $X_p = X(p) \in {\rm Ker}\big(df(p)\big)$. Note that since $n \ge 2$, we have dim$\Big({\rm Ker}\big(df(p)\big)\Big) \ge 2-1 = 1$. For a disc $B(0, \epsilon) \subset \mbb C$, let $Z : B(0, \epsilon) \rightarrow \mbb{C}^n \sm \vert D \vert$ be a holomorphic parametrization of a germ of a leaf of $X$ through $p$, i.e.,
\[
Z'(T) = X\big(Z(T)\big) 
\]
for all $T \in B(0, \epsilon)$ and $Z(0) = p$. We will write $Z_T$ to denote $Z(T)$, and just $u$ instead of $u\big(\vert f(Z_T)\vert^2\big)$. Therefore
\[
Z^{*}(\phi)(T) = \left\{ \big(1 + \left\vert f(Z_T)\right\vert^2 u^{2}\big) \left\vert df(Z_T)\big(X(Z_T)\big)\right\vert^2 + \left\vert X(Z_T)\right\vert^2 \right\} \vert dT\vert^2
\]
and this can be written as $Z^{*}(\phi)(T) = h(T) \left\vert dT\right\vert^2$, where
\[
h(T) = \Big(\big(1 + \left\vert f(Z_T)\right\vert^2 u^{2}\big) \left\vert df(Z_T)\big(X(Z_T)\big)\right\vert^2\Big) + \left\vert X(Z_T)\right\vert^2.
\]
Note that $\partial h(T) = \frac{\partial h (T)}{\partial T} = A_{1}(T) + A_{2}(T) + A_{3}(T)$, where
\begin{align*} 
A_{1}(T) &= \big(u^{2} + 2 \left\vert f(Z_T)\right\vert^2 u u'\big) \left\vert df(Z_T)\big(X(Z_T)\big)\right\vert^2 \left\langle df(Z_T)\big(X(Z_T)\big), f(Z_T)\right\rangle , \\
A_{2}(T) &= \big(1 + \left\vert f(Z_T)\right\vert^2 u^2\big) \left\langle \partial \Big(df(Z_T)\big(X(Z_T)\big)\Big), df(Z_T)\big(X(Z_T)\big)\right\rangle, \\
A_{3}(T) &= \left\langle dX(Z_T)\big(X(Z_T)\big), X(Z_T)\right\rangle.
\end{align*}

\medskip

\no Using the fact that $\overline{h(T)} = h(T)$, it can be checked that $\bar{\partial} h(T) = \overline{\partial h(T)}$. Also $A_1(0) = A_2(0) = \bar{\partial}A_{1}(0) = 0$, and
\begin{align*}
\bar{\partial}A_{2}(T) &= P(T) + \Big( \big(1 + \left\vert f(Z_T)\right\vert^2 u^2 \big) \left\vert \partial \Big(df(Z_T)\big(X(Z_T)\big)\Big)\right\vert^2\Big),\\
\bar{\partial}A_{3}(T) &= \left\vert dX(Z_T)\big(X(Z_T)\big)\right\vert^2.
\end{align*}
where $P(T) = \bar{\partial} \Big(1 + \left\vert f(Z_T)\right\vert^2 u^2\Big) \left\langle \partial \Big(df(Z_T)\big(X(Z_T)\big)\Big), df(Z_T)\big(X(Z_T)\big)\right\rangle$ satisfies $P(0) = 0$.

\medskip

\no Thus, we get that $h(0) = \left\vert X_p\right\vert^2$, $\partial h(0) = \left\langle dX(p)\big(X_p\big), X_p\right\rangle$, $\bar{\partial} h(0) = \left\langle X_p, dX(p)\big(X_p\big)\right\rangle$ and
\[
\bar{\partial}\partial h(0) = \left\vert dX(p)\big(X_p\big)\right\vert^2 + \Big(1 + \vert f(p)\vert^2 u^{2}(\vert f(p)\vert^2)\Big) \left\vert \partial \Big(df(Z_T)\big(X(Z_T)\big)\Big)\big|_{T = 0}\right\vert^2   \ge \left\vert dX(p)\big(X_p\big)\right\vert^2
\]
and this implies that
\[
h(0)\bar{\partial}\partial h(0) - \partial h(0) \bar{\partial} h(0) \ge \left\vert X_p\right\vert^2 \left\vert dX(p)(X_p)\right\vert^2 - \left\vert \left\langle dX(p)(X_p), X_p\right\rangle \right\vert^2 \ge 0
\]
by the Cauchy-Schwarz inequality. Therefore, the Gaussian curvature of the leaves of the foliation induced by $X$ at $p$ is 
\[
\cal{K}^{X}(p) = {-2}(h(0))^{-3} \left\{ h(0)\bar{\partial}\partial h(0) - \partial h(0) \bar{\partial} h(0) \right\} \le 0.
\]

\medskip

\no By \cite{Wu}, there exist a non-singular vector field $Y$ near $p$ such that $Y(p) = V$ and $K(p, V) = \cal{K}^{Y}(p)$. Since the above argument is true for all non-singular holomorphic vector fields near $p$ which equal $V$ at point $p$, it follows that 
\[
K(p, V) \le 0.
\]
\no Now, if ${\rm Rank}\big(df(p)\big) = 0$ at some $p \in \mbb{C}^n \sm \vert D \vert$, then ${\rm Ker}\big(df(p)\big) = \mbb{C}^n$ and therefore $K(p, V) \le 0$ for all $V \in \mbb{C}^n$. Thus 
\[
K^{+}(p) = \sup_{V \in \mbb{C}^n}{K(p, V)} \le 0.
\]

\medskip

\no \textbf{Note:} The above proof is true more generally also. Indeed, let $M$ be a complex manifold of dimension $n \ge 2$ equipped with a metric $\phi$ which has non-positive holomorphic sectional curvature. Let $f : M \rightarrow \mbb{C}$ be a non-constant holomorphic map and assume that $\mbb{C}^{*}$ is equipped with a metric $\psi$. Then the metric $\Phi := f^{*}(\psi) + \phi$ defined on $M \sm \{f = 0\}$ satisfies $K(p, V) \le 0$ for all $p \in M \sm \{f = 0\}$ and $V \in {\rm Ker}\big(df(p)\big)$.

\medskip
 Now let $p \in \mbb{C}^n$ such that ${\rm Rank}\big(df(p)\big) = 1$ and suppose that $f_{z_1}(p) = \frac{\partial f}{\partial z_1}(p) \neq 0$. Then $W(z) = \left(f(z), z_2, \ldots , z_n\right)$ defines a new coordinate system around $p$ with 
\[
W^{-1}(w) = \left(\tilde{g}(w), w_2, \ldots , w_n\right)
\]
for a suitable holomorphic function $\tilde{g}$. Since $W \circ W^{-1} \equiv \text{Id}$, we see that $f_{z_1}(z) \tilde{g}_{w_1}(W(z)) \equiv 1$. For $V = (1,0, \ldots, 0) \in \mbb{C}^n$, 
\[
\left((W^{-1})^{*} \phi\right) (w, V) = (1 + \left\vert w_1\right\vert^2 u^{2}(\vert w_1\vert^2)) \left((W^{-1})^{*} \left\vert df(z)(dz)\right\vert^2 (w,V)\right) + (W^{-1})^{*}(\vert dz\vert^2)(w,V)
\]
and observe that 
\[
\big(W^{-1}\big)^{*}\left(\left\vert df(z)(dz)\right\vert^2\right) (w,V) = \left\vert f_{z_1}\big(W^{-1}(w)\big)\right\vert^2 \left\vert \tilde{g}_{w_1}(w)\right\vert^2  \equiv 1
\]
and $\big(W^{-1}\big)^{*}(\vert dz\vert^2)(w,V) = \left\vert \tilde{g}_{w_1}(w)\right\vert^2 = \left\vert f_{z_1}(W^{-1}(w))\right\vert^{-2}$ . Therefore, we get
\[
\tilde{\phi}(w, V) := \left((W^{-1})^{*} \phi\right) (w, V) = \left(1 + \left\vert w_1\right\vert^2 u^{2}(\vert w_1\vert^2) + \left\vert \tilde{g}_{w_1}(w)\right\vert^2\right).
\]

\no Consider the constant vector field $\tilde{X}(w) \equiv V$ near the point $W(p)$. Take $X := W^{*}(\tilde{X})$ near point $p$ and observe that $X_p = W^{*}(V) \notin {\rm Ker}\big(df(p)\big)$. Let $w = (w_1, w_2, \ldots , w_n) = W(z) \in \mbb{C}^n$ and $Z_T = Z(T) = w + (T, 0, \ldots, 0)$ be a local parametrization of leaf of $\tilde{X}$ passing through $w$. We get
\[
Z^{*}(\tilde{\phi})(T) = \left(\left(1 + \left\vert w_1 + T\right\vert^2 u^{2}\right) + \left\vert \tilde{g}_{w_1}(Z_T)\right\vert^2\right) \left\vert dT\right\vert^2 = \big(A(T) + B(T)\big) \left\vert dT\right\vert^2
\]
where $A(T) = \left(1 + \left\vert w_1 + T\right\vert^2 u^{2}(\vert w_1 + T\vert^2)\right)$ and $B(T) = \left\vert \tilde{g}_{w_1}(Z_T)\right\vert^2$.

\medskip
\no For $(ii)$, suppose that $p \in \mbb{C}^n \sm \vert D\vert$. Fix $w = W(p)$ and observe that $A(T) \vert dT\vert^2$ is the Grauert metric restricted in a neighourhood of $0 \neq f(p) \in \mbb{C}$, therefore the Gaussian curvature is non-positive by \cite{GV}. It is straightforward to check that $\partial B(T) =  \tilde{g}_{w_1 w_1}(Z_T) \overline{\tilde{g}_{w_1}(Z_T)}$, $\bar{\partial} B(T) = \tilde{g}_{w_1}(Z_T) \overline{\tilde{g}_{w_1 w_1}(Z_T)}$ and $\bar{\partial}\partial B(T) = \left\vert \tilde{g}_{w_1 w_1}(Z_T)\right\vert^2$ which tells us that the metric $B(T) \vert dT\vert^2$ has curvature identically $0$ near $T= 0$. Therefore, \cite{Gr2} tells us that the Gaussian curvature of the metric $Z^{*}(\tilde{\phi})$ is non-positive near $0 \in \mbb{C}$. Thus, we get $\cal{K}^{X}(p) \le 0$.

\medskip

\no For $(iii)$, suppose that $p \in \vert D\vert$ and let $w = W(z) \in \mbb{C}^n$ such that $w_1 \neq 0$. As we calculated in $(ii)$, it is clear that, 
\begin{equation}
B(T)\bar{\partial}\partial B(T) - \bar{\partial} B(T) \partial B(T) \equiv 0
\end{equation}
and $\tilde{g}_{w_1}, \tilde{g}_{w_1 w_1}$ are bounded near $W(p)$. Since $A(0) = \left(1 + \left\vert w_1\right\vert^2 u^2 (\vert w_1\vert^2)\right)$, it follows that $\lim_{w_1 \to 0} A(0) = +\infty$, and
\begin{equation}
\lim_{w_1 \to 0} \frac{A(0)}{A(0) + B(0)} = 1.
\end{equation}
Now
\[
\cal{K}^{X}(z) = \cal{K}^{\tilde{X}}(W(z)) = -2 \left(\frac{(A + B)(0) \bar{\partial}\partial (A+B)(0) - \bar{\partial}(A+B)(0) \partial (A+B)(0)}{\left(A(0) + B(0)\right)^3}\right) = -2\frac{N(0)}{D(0)}
\]
where $D(0) = \left(A(0) + B(0)\right)^3$ and
\[
N(0) = \left(A \bar{\partial}\partial A - \partial A \bar{\partial} A\right)(0) + \left(B \bar{\partial}\partial B - \bar{\partial} B \partial B\right)(0) + \left(A \bar{\partial}\partial B + B \bar{\partial}\partial A - \bar{\partial} A \partial B - \partial A \bar{\partial} B\right)(0).
\]
Lemma 5.2 in the last section (using $k = 1$) shows that
\begin{equation}
\lim_{w_1 \to 0} \frac{\bar{\partial} \partial A(0)}{\big(A(0)\big)^3} = \lim_{w_1 \to 0} \frac{\partial A(0)}{\big(A(0)\big)^3} = \lim_{w_1 \to 0} \frac{\bar{\partial} A(0)}{\big(A(0)\big)^3} = 0.
\end{equation}

\no Also, we have
\begin{equation}
-2 \frac{\left(A \bar{\partial}\partial A - \partial A \bar{\partial} A\right)(0)}{\big(A(0)\big)^3} = K_{g}(w_1).
\end{equation}

\no Therefore,
\[
\lim_{z \to p} \cal{K}^{X}(z) = \lim_{w \to W(p)} \left(-2\frac{N(0)}{D(0)}\right) = \lim_{w \to W(p)} \left(-2\frac{N(0)}{\big(A(0)\big)^3}\right)\left(\frac{\big(A(0)\big)^3}{D(0)}\right)
\]
Since $w \to W(p)$ is equivalent to $w_1 \to 0$, an application of $(2.1), (2.2), (2.3) \ {\rm and} \ (2.4)$ gives
\[
\lim_{z \to p} \cal{K}^{X}(z) = \lim_{w_1 \to 0} K_{g}(w_1) + 0 = -4
\]
This completes the proof.

%%%%%%%%%%%%%%%%%%%%%%%
\section{Proof of Theorem 1.2}

Since $H^{1}(U, \mbb{Z}) = 0$, \cite{LS} shows that there exists a continuous map $\psi : \cal{D}^{+}_{P}(U) \rightarrow \cal{O}(U)$ such that ${\sf Div}(\psi(D)) = D$ for all $D \in \cal{D}^{+}_{P}(U)$. Let $\Phi$ denote the complete K\"ahler metric on $U$. Now for any $D \in \cal{D}^{+}_{P}(U)$ and $f \in \cal{O}(U)$ such that ${\sf Div}(f) = D$, we have the following complete K\"ahler metric on $U \sm \vert D \vert$
\[
\phi_{D}(z) = f^{*}(g)(z) + \Phi(z) = \left(1+ \left\vert f(z)\right\vert^2 u^{2}(\vert f(z)\vert^2)\right)f^{*}(\vert dw\vert^2) + \Phi(z)
\]

\no Observe that $f^{*}(\vert dw\vert^2)(z) = \Sigma_{i,k =1}^{n}\frac{\partial f}{\partial z_i}(z) \overline{\frac{\partial f}{\partial z_k}(z)} dz_{i} d\overline{z_{k}}$, and therefore
\[
\phi_{D}(z) =\left( \Sigma_{i,k =1}^{n} \left(1+ \left\vert f(z)\right\vert^2 u^{2}(\vert f(z)\vert^2)\right)\frac{\partial f}{\partial z_i}(z) \overline{\frac{\partial f}{\partial z_k}(z)} dz_{i} d\overline{z_{k}}\right) + \Phi(z).
\]

\no We are given a sequence of non-negative principal divisors $D_{j} \in \cal{D}^{+}_{P}(U)$ such that $D_{j} \to D_{0} \in \cal{D}^{+}_{P}(U)$.

\medskip

\no For $(i)$, define $f_{j} := \psi(D_j)$ and $f_{0} = \psi(D_{0})$. Therefore, we get a sequence of holomorphic functions $\{f_{j}\}_{j \in \mbb{N}} \subset \cal{O}(U)$ such that $f_{j} \to f_{0}$ uniformly on compacts of $U$ with ${\sf Div}(f_j) = D_j$, ${\sf Div}(f_{0}) = D_0$. This gives us corresponding complete K\"ahler metrics $\phi_{j} = f_{j}^{*}(g) + \Phi$ and $\phi_{0} = f_{0}^{*}(g) + \Phi$ on $U \sm \vert D_{j} \vert$ and $U \sm \vert D_{0} \vert$ respectively. Explicitly
\begin{align*}
\phi_{j}(z) &= \Sigma_{i,k =1}^{n} \left(1+ \left\vert f_{j}(z)\right\vert^2 u^{2}(\vert f_{j}(z)\vert^2)\right)\frac{\partial f_{j}}{\partial z_i}(z) \overline{\frac{\partial f_{j}}{\partial z_k}(z)} dz_{i} d\overline{z_{k}} + \Phi(z),\\
\phi_{0}(z) &= \Sigma_{i,k =1}^{n} \left(1+ \left\vert f_{0}(z)\right\vert^2 u^{2}(\vert f_{0}(z)\vert^2)\right)\frac{\partial f_{0}}{\partial z_i}(z) \overline{\frac{\partial f_{0}}{\partial z_k}(z)} dz_{i} d\overline{z_{k}} + \Phi(z).
\end{align*}

\no Since $f_{j} \to f_{0}$ uniformly on compacts of $U$, all derivatives of $f_{j}$ will also converge uniformly to the corresponding derivative of $f_{0}$ on compacts of $U$. Since $u : (0, \infty) \rightarrow \mbb{R}$ is real analytic, it follows that $u(\vert f_{j}\vert^2) \to u(\vert f_{0}\vert^2)$ uniformly on compacts of $U \sm \vert D_{0} \vert$. Therefore, $\phi_{j} \to \phi_{0}$ uniformly on compacts of $U \sm \vert D_{0} \vert$.

\medskip
\no For $(ii)$, suppose we are given a complete K\"ahler metric on $U \sm \vert D_{0} \vert$ of the form $\xi = \tilde{f_{0}}^{*}(g) + \Phi$ where ${\sf Div}(\tilde{f_{0}}) = D_{0}$. Let $f_{0} = \psi(D_{0})$  and observe that $h(z) := \tilde{f_{0}}(z)/f_{0}(z) \in \cal{O}^{*}(U)$ since ${\sf Div}(f_{0}) = {\sf Div}(\tilde{f_{0}}) = D_{0}$. Define 
\[
\tilde{\psi} : \cal{D}^{+}_{P}(U) \rightarrow \cal{O}(U)
\]
by $\tilde{\psi}(D)(z) := h(z)[\psi(D)(z)]$. Clearly $\tilde{\psi}$ is continuous and $\tilde{\psi}(D_{0}) = h \psi(D_{0}) = h f_{0} = \tilde{f_{0}}$. Also,
\[
{\sf Div}\left(\tilde{\psi}(D)\right) = {\sf Div}\big(h \psi(D)\big) = {\sf Div}\big(\psi(D)\big) = D
\]
for all $D \in \cal{D}^{+}_{P}(U)$. So we can define $\tilde{f_{j}} := \tilde{\psi}(D_{j})$. Here $\tilde{f_{j}} \to \tilde{f_{0}}$ uniformly on compacts of $U$ and ${\sf Div}(\tilde{f_{j}}) = D_{j}$. So the corresponding complete K\"ahler metrics $\xi_{j} := \tilde{f_{j}}^{*}(g) + \Phi$ converges uniformly to the metric $\xi = \tilde{f_{0}}^{*}(g) + \Phi$ on compacts of $U \sm \vert D_{0} \vert$.

%%%%%%%%%%%%%%%%%%%%%%%%%%%%%%%
\section{Proof of Theorem 1.3}

We are given complete K\"ahler metrics $\phi_{j} = f_{j}^{*}(g) + \Phi$ and $\phi_{0} = f_{0}^{*}(g) + \Phi$ on $U \sm \vert D_{j} \vert$ and $U \sm \vert D_{0} \vert$ respectively, where $f_{j}, f_{0} \in \cal{O}(U)$ are such that $f_{j} \to f_{0}$ uniformly on compacts of $U$.

\medskip
\no $(i)$ Let $p \in U \sm \vert D_{0} \vert$. Since $D_{j} \to D_{0}$, choose a neighourhood $U_{p}$ of $p$ relatively compact in $U \sm \vert D_{0} \vert$ such that $\overline{U}_{p} \subset U \sm \vert D_{j} \vert$ for large enough $j$. Let $X$ be a non-singular holomorphic vector field near $p$. For $q \in U_{p}$, consider a parametrization $Z : B(0, \epsilon) \rightarrow U_{p}$ such that $Z(0) = q$, $Z_T = Z(T)$ and $dZ(T)/dT = X(Z_T)$. Then  $Z^{*}(\phi_{j}) = h_{j}(T)\vert dT\vert^2$ and $Z^{*}(\phi_{0}) = h_{0}(T)\vert dT\vert^2$, where
\[
h_j(T) = \Sigma_{i,k =1}^{n} \left(1+ \left\vert f_{j}(Z_T)\right\vert^2 u^{2}(\vert f_{j}(Z_T)\vert^2)\right)\frac{\partial f_{j}}{\partial z_i}(Z_T) \overline{\frac{\partial f_{j}}{\partial z_k}(Z_T)} X_{i}(Z_T) \overline{X_{k}(Z_T)} + \tilde{h}(T)
\]
and
\[
h_0(T) = \Sigma_{i,k =1}^{n} \left(1+ \left\vert f_{0}(Z_T)\right\vert^2 u^{2}(\vert f_{0}(Z_T)\vert^2)\right)\frac{\partial f_{0}}{\partial z_i}(Z_T) \overline{\frac{\partial f_{0}}{\partial z_k}(Z_T)} X_{i}(Z_T) \overline{X_{k}(Z_T)} + \tilde{h}(T)
\]
with $\tilde h = Z^{\ast}\Phi$. Further, recall that
\[
\cal{K}^{X}_{\phi_j}(q) = {-2}(h_{j}(0))^{-3} \left\{ h_{j}(0) \partial \bar{\partial}h_{j}(0) - \partial h_{j}(0) \bar{\partial} h_{j}(0) \right\}.
\]
\no Since $f_{j} \to f_{0}$ uniformly on compacts of $U$, $\partial^{i}\bar{\partial}^{k}f_{j}(Z_T) \to \partial^{i}\bar{\partial}^{k} f_{0}(Z_T)$ for all $T$ in a neighourhood of $0 \in \mbb{C}$ and for all $i, k \ge 0$. Similarly, $\vert f_{j}(Z_T)\vert^2$, $u(\vert f_{j}(Z_T)\vert^2)$ and their higher derivatives also converge uniformly on a neighourhood of $0 \in \mbb{C}$. Thus
\[
\partial^{i}\bar{\partial}^{k}h_{j}(T) \to \partial^{i}\bar{\partial}^{k}h_{0}(T)
\]
for all $0 \le i,k \le 1$ and for all $T$ in a neighourhood of $0 \in \mbb{C}$. Therefore, $\cal{K}^{X}_{\phi_j}(q) \to \cal{K}^{X}_{\phi_0}(q)$ for all $q \in U_{p}$.

\medskip

\no $(ii)$ Let $p \in U \sm \vert D_{0} \vert$ and $V \in \mbb{C}^n$. Clearly $p \in U \sm \vert D_j \vert$ for large enough $j$. Let $X$ be a non-singular vector field near $p$ such that $X(p) = V$ and $K_{\phi_0}(p, V) = \cal{K}^{X}_{\phi_0}(p)$. By $(i)$, $\cal{K}^{X}_{\phi_j}(p) \to \cal{K}^{X}_{\phi_0}(p)$. Therefore, for $\epsilon > 0$, there exist $N \ge 1$ such that 
\[
\cal{K}^{X}_{\phi_0}(p) \le \cal{K}^{X}_{\phi_j}(p) + \epsilon \le K_{\phi_j}(p, V) + \epsilon
\]
for all $j \geq N$ and this gives $K_{\phi_0}(p, V) \le K_{\phi_j}(p, V) + \epsilon$. Hence, $K_{\phi_0}(p, V) \le \liminf_{j \to \infty}{K_{\phi_j}(p, V)} + \epsilon$ for all $\epsilon >0$ and consequently, $K_{\phi_0}(p, V) \le \liminf_{j \to +\infty}{K_{\phi_j}(p, V)}$.

%%%%%%%%%%%%%%%%%%%%%
\section{The case $n=1$}

\no For a non-constant holomorphic function $f: \mbb{C} \rightarrow \mbb{C}$, consider the complete K\"ahler metric on $\mbb{C} \sm \{f = 0\}$ given by $\phi(z) = f^{*}(g) + \vert dz \vert^2$. We have
\[
\phi(z) = h(z) \vert dz \vert^2 = \left(\vert f'(z) \vert^2 \left(1 + \left\vert f(z)\right\vert^2 u^{2}(\vert f(z) \vert^2)\right) + 1\right)\vert dz \vert^2
\]
Denote the Gaussian curvature of the metric $\phi$ by $K$. 

\begin{thm}
Let $(f, \phi, K)$ be as above and $p \in \mbb{C}$. 
\begin{itemize}
\item[(i)] If $p \in \mbb{C} \sm \{f = 0\}$, then $K(p) \le 0$. In addition, if $f'(p) = 0$ then $K(p) = 0$ if and only if $f''(p) = 0$.
\item[(ii)] If $p \in \{f = 0\}$, then 
\[
\lim_{z \to p}{K(z)} = -4
\]
\end{itemize}
\end{thm}

\noindent The proof relies on the following observations.

\begin{lem} 
For an integer $k \ge 1$, consider the metric 
\[
\phi_{k} (z) = k^2 \left\vert z \right\vert^{2(k-1)} \left(1 + \left\vert z\right\vert^{2k} u^2(\vert z\vert^{2k})\right)\vert dz\vert^2 = h_{k}(z) \vert dz \vert^2
\]
on $\mbb{C}^{*}$.
The Gaussian curvature of this metric $K_k$ satisfies $K_k(z) = K_{g}(z^k)$ for $z \in \mbb{C}^{*}$ and
\[
\lim_{z \to 0}{K_{k}(z)} = -4
\]
Also, the function $h_k$ satisfies $\lim_{z \to 0}{h_k(z)} = +\infty$ and 
\[
\frac{\partial h_k(z)}{\big(h_k(z)\big)^3}, \frac{\bar{\partial} h_k(z)}{\big(h_k(z)\big)^3}, \frac{\bar{\partial}\partial h_k(z)}{\big(h_k(z)\big)^3} \rightarrow 0
\]
as $z \to 0$. 
\end{lem}

\begin{proof}
Since $\phi_k(z) = h_k(z) \vert dz\vert^2$, the curvature is given by
\[
K_k(z) = -2 \left(\frac{h_k(z) \bar{\partial}\partial h_k(z) - \partial h_k(z) \bar{\partial} h_k(z)}{\big(h_k(z)\big)^3}\right)
\]
for $z \in \mbb{C}^{*}$. We will write just $u$ instead of $u(\vert z\vert^{2k})$. Observe that:
\[
\partial h_k(z) = \overline{\bar{\partial} h_k(z)} = k^2 \bar{z}\left((k-1)\left\vert z\right\vert^{2k-4} + (2k-1) \left\vert z\right\vert^{4k-4} u^2 + 2k \left\vert z\right\vert^{6k-4} u u'\right)
\]
\[
\bar{\partial}\partial h_k(z) = \vert z\vert^{2k-4} \big( k^2 (k-1)^2  + k^2 (2k-1)^2 \vert z\vert^{2k} u^2 + 2k^3(5k-2)\vert z\vert^{4k} u u' + 2k^4 \vert z\vert^{6k} \big((u')^2 + uu''\big)\big)
\]
Upon simplification, we get
\[
h_k(z) \bar{\partial}\partial h_k(z) - \partial h_k(z) \bar{\partial} h_k(z) = k^6 \left\vert z\right\vert^{6k-6} M(\vert z\vert^{2k})
\]
where $M : (0, +\infty) \rightarrow \mbb{R}$  is given by
\[
M(t) = u^2 + 6t uu' + 2t^2(u')^2 +2t^2 uu'' + 2t^2 u^3 u' - 2t^3 u^2 (u')^2 + 2t^3 u^3 u''
\]
Now \cite{GV} tells us that 
\[
K_g(z) = -2\frac{M(\vert z\vert^2)}{\left(1 + \left\vert z\right\vert^2 u^2(\vert z\vert^2)\right)^3}
\]
for $z \in \mbb{C}^{*}$. Now note that
\[ 
K_k(z) = -2 \frac{k^6 \left\vert z\right\vert^{6k-6} M(\vert z\vert^{2k})}{k^6 \left\vert z\right\vert^{6k-6}\left(1 + \left\vert z\right\vert^{2k} u^2 (\vert z\vert^{2k})\right)^3} = -2\frac{M(\vert z^k\vert^2)}{\left(1 + \left\vert z^k\right\vert^2 u^2 (\vert z^k\vert^2)\right)^3} = K_g(z^k)
\]
This shows that for $k \in \mbb{N}$,  $z \in \mbb{C}^*$, $K_k(z) = K_g(z^k)$ and 
\[
\lim_{z \to 0}{K_k(z)} = \lim_{z \to 0}{K_g(z^k)} = -4
\]

\medskip 

To focus on the function $h_k(z)$, observe that
\[
\lim_{z \to 0} h_k(z) = \lim_{t \to 0^{+}}{k^2 t^{k-1} \left(1 + t^k u^2(t^k)\right)} = \lim_{t \to 0^{+}} \frac{k^2 t^{2k-1} (t^k -1)^2}{t^{2k} (\log{t^k})^2} = \lim_{t \to 0^{+}} \frac{k^2 (t^k -1)^2}{t(\log{t^k})^2} = +\infty
\]
Now substituting for $u, u', u''$, simplifying the resulting expression and using the fact that $\lim_{t \to 0^{+}} t (\log{t}) = 0$, we get
\[
\frac{\left\vert z\right\vert^{2k-4}}{\left(h_k(z)\right)^3}, \frac{\left\vert z\right\vert^{4k-4} u^2}{\left(h_k(z)\right)^3}, \frac{\left\vert z\right\vert^{6k-4} u u'}{\left(h_k(z)\right)^3}, \frac{\left\vert z\right\vert^{8k-4} \left((u')^2 + uu''\right)}{\left(h_k(z)\right)^3} \to 0
\]
as $z \to 0$. It is now evident that
\[
\lim_{z \to 0} \frac{\partial h_k(z)}{\left(h_k(z)\right)^3} = \lim_{z \to 0} \frac{\bar{\partial} h_k(z)}{\left(h_k(z)\right)^3} = \lim_{z \to 0} \frac{\bar{\partial} \partial h_k(z)}{\left(h_k(z)\right)^3} = 0
\]

\end{proof}

Now we can prove Theorem 5.1. 

\begin{proof} For $z \in \mbb{C} \sm \{f = 0\}$, the curvature is given by
\[
K(z) = -2 \left(\frac{h(z)\partial \bar{\partial}h(z) - \partial h(z) \bar{\partial}h(z)}{\big(h(z)\big)^3}\right).
\]

\no For $(i)$, suppose that $f'(p) \neq 0$. Let $\hat{g}$ be a local inverse of $f$ near $p$, that is $\hat{g} \circ f(z) = z$ for $z$ in a neighourhood $U_p$ around $p$. It can be checked that $\hat{g}'(w) = \big(f'(\hat{g}(w))\big)^{-1}$ and $\hat{g}''(w) = -f''(\hat{g}(w))\big(f'(\hat{g}(w))\big)^{-2}$. In these new coordinates, we have
\[
\hat{g}^{*}(\phi)(w) = \left(1 + \left\vert \hat{g}'(w) \right\vert^{-2} \left(1 + \left\vert w \right\vert^2 u^2(\vert w \vert^2)\right)\right)\vert \hat{g}'(w) \vert^2 \vert dw \vert^2  
\]
\[
\hat{g}^{*}(\phi)(w) = \left(\left\vert \hat{g}'(w) \right\vert^2 + \left(1 + \left\vert w \right\vert^2 u^2(\vert w \vert^2)\right)\right)\vert dw \vert^2 = \phi_1(w) + \phi_2(w)
\]
where $\phi_2(w)$ is the Grauert metric on $\mbb{C}^{*}$ and $\phi_1(w) = \left\vert \hat{g}'(w) \right\vert^2 \vert dw \vert^2$. By \cite{GV}, $K_g(w) \le 0$ for all $w \in \mbb{C}^{*}$ and note that $K_1$, the curvature of the metric $\phi_1$ satisfies $K_1 \equiv 0$. Now \cite{Gr2} tells us that the curvature of sum of these two metrics is also non-positive in $U_p$, that is $K(p) \le 0$.

\medskip
\noindent To continue, note that if $f'(p) = 0$, then $h(p) =1$. Also
\[
\partial h(z) = \left(\left\vert f'(z) \right\vert^2 \partial \big(1 + \left\vert f \right\vert^2 u^2\big) + f''(z) \overline{f'(z)}\big(1 + \left\vert f \right\vert^2 u^2\big)\right)
\]
\[
\bar{\partial}\partial h = \left( \bar{\partial}\Big(\left\vert f'\right\vert^2 \partial (1+ \left\vert f\right\vert^2 u^2)\Big) + f'' \overline{f'}\bar{\partial} \big(1 + \left\vert f\right\vert^2 u^2\big) + \left\vert f''\right\vert^2 \big(1 + \left\vert f\right\vert^2 u^2\big)\right)
\]
which gives us $\partial h(p) = \bar{\partial} h(p) = 0$ and $\bar{\partial}\partial h(p) = \left\vert f''(p) \right\vert^2 \Big(1 + \left\vert f(p)\right\vert^2 u^2(\vert f(p)\vert^2)\Big)$. Therefore we get
\[
K(p) = -2\left\vert f''(p) \right\vert^2 \Big(1+ \left\vert f(p) \right\vert^2 u^2(\vert f(p)\vert^2)\Big)
\]
and hence $K(p) = 0$ if and only if $f''(p) = 0$.

\medskip

\no For $(ii)$, let $p \in \mbb{C}$ be such that $f(p) =  0$. By a translation, assume that $p = 0$. Let $k = {\rm Ord}_{0}(f)$. Then there exists a holomorphic map $f_1$ such that $f_1(0) \neq 0$ and $f(z) = z^k f_1(z)$. Now since $f_1(0) \neq 0$, there exist a holomorphic map $f_2$ defined near $0 \in \mbb{C}$ such that $f_2(0) \neq 0$ and $f_1(z) = \big(f_2(z)\big)^k$. So for $z $ in a neighourhood of $0 \in \mbb{C}$, we get
\[
f(z) = \big(z f_2(z)\big)^k
\]
Define $W(z) = z f_2(z)$ for $z$ near origin and observe that $W'(0) = f_2(0) \neq 0$, that is $W$ is a biholomorphism on a neighourhood around $0 \in \mbb{C}$. Therefore there exist a holomorphic function $g_1$ defined near $0 \in \mbb{C}$ such that $g_1(0) = 0$ and $W^{-1}(w) = g_1(w)$. Since $f'(z) = k\big(W(z)\big)^{k-1}W'(z)$, it follows that $f'\big(W^{-1}(w)\big) = k w^{k-1} \left((W^{-1})' (w)\right)^{-1}$. Now
\begin{align*}
\left((W^{-1})^{*} \phi\right)(w) &= \left(1 + k^2 \left\vert w\right\vert^{2(k-1)} \left\vert g_{1}'(w)\right\vert^{-2} \big(1 + \left\vert w\right\vert^{2k} u^2(\vert w\vert^{2k})\big)\right) \left\vert g_{1}'(w)\right\vert^2 \vert dw\vert^2\\
                                     &= \left( \left\vert g_{1}'(w)\right\vert^2 + k^2 \left\vert w\right\vert^{2(k-1)} \big(1 + \left\vert w\right\vert^{2k} u^2(\vert w\vert^{2k})\big) \right) \vert dw\vert^2. 
\end{align*}

\medskip
Therefore, $\left((W^{-1})^{*} \phi\right)(w) = \left\vert g_{1}'(w)\right\vert^2 \vert dw\vert^2 + h_{k}(w) \vert dw\vert^2$. Observe that $\left\vert g_{1}'(w)\right\vert^2 \vert dw\vert^2$ has identically vanishing Gaussian curvature. Now using Lemma 5.2 and the fact that $g_{1}'(w), g_{1}''(w)$ are bounded in a neighourhood of $w = 0 \in \mbb{C}$, we get
\[
\lim_{z \to 0} K(z) = \lim_{w \to 0} -2 \left(\frac{h_{k}(w) \bar{\partial} \partial h_{k}(w) - \partial h_{k}(w) \bar{\partial} h_{k}(w)}{\big(h_{k}(w) + \vert g_{1}'(w)\vert^2\big)^3}\right) = \lim_{w \to 0} K_{k}(w) = -4
\]
This completes the proof.
\end{proof}

%%%%%%%%%%%%%%%%%%%%%%

\end{document}